\documentclass[12pt]{amsart}
\usepackage{amssymb,graphics}
\allowdisplaybreaks

\title[A substitute for Kazhdan's property (T)]{A substitute for Kazhdan's property (T) for universal non-lattices}
\author[N. Ozawa]{Narutaka Ozawa}
\address{RIMS, Kyoto University, \mbox{606-8502} Japan}
\email{narutaka@kurims.kyoto-u.ac.jp}

\date{\today}

\newtheorem*{mainthm}{Main Theorem}
\newtheorem{corA}{Corollary}
 
\newtheorem{thm}{Theorem}
\newtheorem{prop}[thm]{Proposition}
\newtheorem{cor}[thm]{Corollary}
\newtheorem{lem}[thm]{Lemma}

\newcommand{\IC}{\mathbb C}
\newcommand{\IM}{\mathbb M}
\newcommand{\IN}{\mathbb N}
\newcommand{\IP}{\mathbb P}

\newcommand{\IR}{\mathbb R}
\newcommand{\IZ}{\mathbb Z}
\newcommand{\cA}{\mathcal A}

\newcommand{\cH}{\mathcal H}
\newcommand{\cR}{\mathcal R}
\newcommand{\cS}{\mathcal S}

\newcommand{\cZ}{\mathcal Z}
\newcommand{\bH}{\mathbf H}
\newcommand{\Ga}{\Gamma}
\newcommand{\ve}{\varepsilon}
\newcommand{\vp}{\varphi}
\newcommand{\ux}{\bar{x}}
\newcommand{\uy}{\bar{y}}
\newcommand{\uz}{\bar{z}}
\newcommand{\re}{\mathrm e}
\newcommand{\rf}{\mathrm f}
\newcommand{\rE}{\mathrm E}
\DeclareMathOperator{\diag}{diag}

\newcommand{\ab}{\mathrm{ab}}
\newcommand{\her}{\mathrm{her}}
\DeclareMathOperator{\lh}{span}

\DeclareMathOperator{\Sym}{Sym}
\DeclareMathOperator{\tr}{tr}

\DeclareMathOperator{\EL}{EL}
\DeclareMathOperator{\SL}{SL}
\DeclareMathOperator{\Sq}{Sq}
\DeclareMathOperator{\Adj}{Adj}
\DeclareMathOperator{\Op}{Op}
\DeclareMathOperator{\ran}{ran}
\newcommand{\ip}[1]{\mathopen{\langle}#1\mathclose{\rangle}}
\newcommand{\ca}{$\mathrm{C}^*$-alge\-bra}
\newcommand{\cst}{\mathrm{C}^*}

\def\pprec{\mathrel{\scalebox{.8}[1]{$\prec$}\mkern-4mu%
  \scalebox{.4}[1]{$\prec$}\mkern-5.5mu\scalebox{.4}[1]{$\prec$}}}

\begin{document}
\begin{abstract}
The well-known theorem of Shalom--Vaserstein and Ershov--Jaikin-Zapirain 
states that the group $\mathrm{EL}_n(\mathcal{R})$, generated by 
elementary matrices over a finitely generated commutative ring $\mathcal{R}$, 
has Kazhdan's property (T) as soon as $n\geq3$. 
This is no longer true if the ring $\mathcal{R}$ is replaced by 
a commutative rng (a ring but without the identity) due to nilpotent quotients 
$\mathrm{EL}_n(\mathcal{R}/\mathcal{R}^k)$. In this paper, we prove that 
even in such a case the group $\mathrm{EL}_n(\mathcal{R})$ satisfies a certain 
property that can substitute property (T), provided that $n$ is large enough. 
\end{abstract}
\maketitle

\section{Introduction}
In this paper, we continue and extend the scope of the study 
of \cite{kno,kkn,nt2,nitsche,ncrag}, which develops the way of 
proving Kazhdan's property (T) via sum of squares methods. 
See \cite{bhv} for a comprehensive treatment of property (T).
Let $\Ga=\langle S \rangle$ be a group 
together with a finite symmetric generating subset $S$. 
We denote by $\IR[\Ga]$ the real group algebra with 
the involution $\ast$ that extends the inverse $\ast\colon x\mapsto x^{-1}$ on $\Ga$. 
The positive elements in $\IR[\Ga]$ are sums of (hermitian) squares, 
\[
\Sigma^2\IR[\Ga] := \{ \sum_i \xi_i^*\xi_i : \xi_i\in\IR[\Ga]\}
\]
and the combinatorial Laplacian is 
\[
\Delta := \frac{1}{2}\sum_{s\in S} (1-s)^*(1-s)
 = |S| - \sum_{s\in S} s \in \Sigma^2\IR[\Ga] .
\]
It is proved in \cite{ncrag} that the group $\Ga$ has property (T) if and only if 
there is $\ve>0$ that satisfies the inequality 
\[
\Delta^2 - \ve\Delta \in \Sigma^2\IR[\Ga].
\]
Property (T) for the so-called \emph{universal lattice} $\EL_n(\IZ[t_1,\ldots,t_d])$, 
$n\geq3$, is proved by Shalom (\cite{shalom}), 
Shalom--Vaserstein (\cite{vaserstein}), and 
Ershov--Jaikin-Zapirain (\cite{ej}).
See also \cite{mimura} for a simpler proof and \cite{kn,kno} 
for partial results. 
All the proofs (save for \cite{kno}) rely on relative property (T) 
of certain semi-direct products. 
Our interest in this paper is in the infinite index subgroup 
$\EL_n(\IZ\langle t_1,\ldots,t_d\rangle)$ of 
$\EL_n(\IZ[t_1,\ldots,t_d])$.
Here $\cR:=\IZ\langle t_1,\ldots,t_d\rangle$ is the commutative 
\emph{rng} (i.e., a ring, but without assuming the existence of the identity; 
$\cR$ is an ideal in the unitization $\cR^1$) of polynomials 
in $t_1,\ldots,t_d$ with zero constant terms and 
$\EL_n(\cR)\subset\SL_n(\cR^1)$ denotes the group generated by 
the elementary matrices over the rng $\cR$. 
The elementary matrices are those $e_{i,j}(r)\in \SL_n(\cR^1)$ 
with $1$'s on the diagonal, $r\in \cR$ in the $(i,j)$-th entry, 
and zeros everywhere else. 
The group $\EL_n(\cR)$ does not have property (T), 
because it has infinite nilpotent quotients $\EL_n(\cR/\cR^k)$. 
The group does not seem to admit a good analogue of relative property (T) 
phenomenon, either. 
Still, we prove via sum of squares methods that 
$\EL_n(\cR)$ satisfies a property 
that can substitute property (T). 

\begin{mainthm}
Let $d\in\IN$ and consider the commutative rng $\cR:=\IZ\langle t_1,\ldots,t_d \rangle$. 
Then there are $n_0\in\IN$ and $\ve>0$ such that for every $n\geq n_0$, 
the combinatorial Laplacians 
\[
\Delta:=\sum_{i\neq j}\sum_{r=1}^d (1-e_{i,j}(t_r))^*(1-e_{i,j}(t_r))
\]
for $\EL_n(\cR)$ and 
\[
\Delta^{(2)} := \sum_{i\neq j}\sum_{r,s=1}^d (1-e_{i,j}(t_rt_s))^*(1-e_{i,j}(t_rt_s))
\]
for $\EL_n(\cR^2)$ satisfy 
\[
 \Delta^2 - n \ve  \Delta^{(2)} \in \overline{\Sigma^2\IR[\EL_n(\cR)]}.
\]
\end{mainthm}

Here $ \overline{\Sigma^2\IR[\Ga]}$ denotes the archimedean closure 
of $\Sigma^2\IR[\Ga]$ (see Section~\ref{sec:prelim}). 
An upper bound for $n_0$ in Main Theorem is in principle explicitly calculable, 
but we do not attempt to do that (nor attempt to optimize the proof 
for a better estimate). 
We conjecture\footnote{NB: As the author is lame at computer, 
no computer experiments have been carried out.} 
that Main Theorem holds true with $n_0=3$ 
(in particular $n_0$ should not depend on $d$). 
Our proof is inspired by the work of Kaluba--Kielak--Nowak (\cite{kkn}) 
that proves property (T) for $\mathrm{Aut}( F_d )$ for $d\geq 5$ via 
computer calculations and an ingenious idea on stability. 
Our proof does not rely on computers, 
but instead on analysis by Boca and Zaharescu (\cite{bz}) 
on the almost Mathieu operators in the rotation {\ca}s.  
In fact, there is no known method of rigorously proving 
the result like Main Theorem by computers. 
This is because the conclusion is \emph{analytic} in nature---the archimedean 
closure is indispensable. 
See discussions in Section~\ref{sec:realgrpalg}. 

The above theorem has a couple of corollaries. 
The first one is a reminiscent of one of the standard 
definitions of property (T) (see Definition 1.1.3 in \cite{bhv}). 

\begin{corA}\label{cor:kaz}
For every $d$, if $n$ is large enough, then 
for every $\kappa>0$ there is $\delta>0$ satisfying the following property. 
For every orthogonal representation $\pi$ 
of $\EL_n(\IZ\langle t_1,\ldots,t_d \rangle)$ on a Hilbert 
space $\cH$ and every unit vector $v\in\cH$ with 
$\max_{i,j,r}\| v -\pi(e_{i,j}(t_r)) v\|\le\delta$, 
there is a vector $w\in\cH$ such that 
$\| v - w \| \le \kappa$ and 
\[
\lim_{l\to\infty} \max_{i,j,r} \| w - \pi(e_{i,j}(t_r^l)) w\| =0.
\]
\end{corA}

We remark that a certain strengthening of the above corollary 
does not hold. Namely, there is an orthogonal representation $\pi$ 
of $\EL_n(\IZ\langle t_1,\ldots,t_d \rangle)$ that simultaneously 
admits asymptotically invariant vectors $v_k$ and a sequence 
$x_l \in \EL_n(\IZ\langle t_1^l,\ldots,t_d^l \rangle)$ 
with $\pi(x_l) \to 0$ in the weak operator topology.

\begin{corA}\label{cor:tau}
For every $d$, if $n$ is large enough, then the group 
$\EL_n(\IZ\langle t_1,\ldots,t_d \rangle)$ has property $(\tau)$ 
with respect to the finite quotients of the form $\EL_n(\cS)$, 
where $\cS$ is a finite unital quotients of $\IZ\langle t_1,\ldots,t_d \rangle$. 
\end{corA}

Property $(\tau)$ is a generalization of property (T) for finite quotients. 
See Section~\ref{sec:tau} for the definition 
and the proofs of the above corollaries. 
Corollary~\ref{cor:tau} says $\{ \EL_n(\cS) : \cS\}$ forms an expander family 
with respect to elementary generating subsets of fixed size. 
The novel point compared to the previously known case of 
the universal lattice (\cite{kn}) is that the generating subsets of the finite commutative 
rings $\cS$ need not contain the unit although $\cS$ are assumed unital. 
For example, for $n$ large enough, the Cayley graphs of 
$\SL_n(\IZ/q\IZ)$ w.r.t.\ the generating subsets $\{ e_{i,j}(p) : i\neq j\}$ 
form an expander family as relatively prime pairs $(p,q)$ vary. 
The study of the expander property for $\SL_n(\IZ/q\IZ)$ and alike is 
a very active area. See \cite{bl,helfgott,kowalski} for recent surveys on this.

\subsection*{Acknowledgments}
The author is grateful to Professor Marek Kaluba for communications 
around the material of Section~\ref{sec:realgrpalg} and to 
Professor Nikhil Srivastava 
on the almost Mathieu operators. 
The author was partially supported by JSPS KAKENHI Grant No.\ 20H01806. 
\section{Preliminary}\label{sec:prelim}

Let $\Ga=\langle S\rangle$ be a group together with a 
finite symmetric generating subset $S$. 
We denote by $\IR[\Ga]$ the real group algebra 
with the involution $\ast$ which is the linear 
extension of $x^* := x^{-1}$ on $\Ga$. 
The identity element of $\Ga$ as well as $\IR[\Ga]$ is simply denoted by $1$. 
Recall the positive cone of \emph{sums of (hermitian) squares} is given by 
\[
\Sigma^2\IR[\Ga] := \{ \sum_i \xi_i^*\xi_i : \xi_i\in\IR[\Ga]\}
 \subset\IR[\Ga]^\her := \{ \xi\in\IR[\Ga] : \xi=\xi^*\}.
\]
The elements in $\Sigma^2\IR[\Ga]$ are considered positive.
For $\xi,\eta\in\IR[\Ga]^\her$, we write $\xi\preceq \eta$ 
if $\eta-\xi \in \Sigma^2\IR[\Ga]$.
It is obvious that $\xi\succeq0$ implies $\xi\geq0$ in the full group {\ca} $\cst[\Ga]$, 
that is to say, $\pi(\xi)$ is positive selfadjoint for every orthogonal (or unitary)
representation $\pi$ of $\Ga$ on a real (or complex) Hilbert space $\cH$. 
The converse is true up to the \emph{archimedean closure}: 
\[
\overline{\Sigma^2\IR[\Ga]}
:=\{ \xi\in \IR[\Ga] : \forall \ve>0\ \xi+\ve\cdot 1\succeq0\}
=\{ \xi\in\IR[\Ga] : \xi\geq0 \mbox{ in } \cst[\Ga]\}. 
\]
See, e.g., \cite{cimpric,cec,schmudgen} for this. 
On this occasion, we remind the basic fact that 
$0\preceq\xi\preceq\eta$ (or $0\le\xi\le\eta$) 
need not imply $0\le\xi^2\le\eta^2$. 
Note that since any orthogonal representation of $\Ga$ dilates 
to an orthogonal representation of any supergroup $\Ga_1\geq\Ga$ 
by induction (i.e., $\cst[\Ga]\subset\cst[\Ga_1]$ in short), 
whether $\xi\geq0$ or not does not depend on the ambient group. 
The same holds true for $\xi\succeq0$, by the coset decomposition.  
The \emph{combinatorial Laplacian}, w.r.t.\ the (symmetric) generating subset $S$,
\[
\Delta:=\frac{1}{2}\sum_{s\in S} (1-s)^*(1-s)=|S|-\sum_{s\in S} s 
\]
satisfies, for every orthogonal representation $(\pi,\cH)$ and a vector $v\in\cH$, 
\[
\ip{\pi(\Delta)v,v} = \frac{1}{2}\sum_{s\in S}\| v-\pi(s)v\|^2.
\]
%This is why one can analyze the behavior of asymptotic invariant vectors 
%through the spectral analysis of $\Delta$. 
%
%
\section{Proof of Main Theorem, Prelude}\label{sec:prelude}

For any rng $\cR$, we denote by $\EL_n(\cR)\subset\SL_n(\cR^1)$ 
the group generated by the elementary matrices over the rng $\cR$. 
The elementary matrices are those $e_{i,j}(r)\in \SL_n(\cR^1)$ 
with $1$'s on the diagonal, $r\in \cR$ in the $(i,j)$-th entry ($i\neq j$), 
and zeros everywhere else. 
They satisfy the Steinberg relations:
\begin{itemize}
\item $e_{i,j}(r) e_{i,j}(s)=e_{i,j}(r+s)$;
\item $[e_{i,j}(r),e_{j,k}(s)]=e_{i,k}(rs)$ if $i\neq k$;
\item $[e_{i,j}(r), e_{k,l}(s)]=1$ if $i\neq l$ and $j\neq k$.
\end{itemize}
We note that every rng homomorphism $\cR\to\cS$ 
induces by entrywise operation a group homomorphism 
$\EL_n(\cR)\to\EL_n(\cS)$ and that $\EL_n(\cR/\cR^k)$ 
is nilpotent for every $k$, where $\cR^k:=\lh\{ r_1\cdots r_k :  r_i\in \cR\}$.
To ease notation, we will write 
\[
E_{i,j}(r) := (1-e_{i,j}(r))^*(1-e_{i,j}(r)) = 2- e_{i,j}(r) - e_{i,j}(r)^* \in \IR[\EL_n(\cR)].
\]

We now consider the case $\cR=\IZ\langle t_1,\ldots,t_d \rangle$ and 
start proving Main Theorem. 
Recall that the combinatorial Laplacians w.r.t.\ the 
generating subset $\{ e_{i,j}(\pm t_r)\}$ are given by 
\[
\Delta_n := \sum_{i\neq j}\sum_{r=1}^d E_{i,j}(t_r) 
\ \mbox{ and }\ 
\Delta_n^{(2)} := \sum_{i\neq j} \sum_{r,s=1}^d  E_{i,j}(t_rt_s).
\]
We follow the idea of \cite{kkn} about the stability w.r.t.\ $n$ 
of the relation like $\Delta^{(2)}_n \ll \Delta^2_n$.  
Here $\xi\ll\eta$ means that $\xi\le R\eta$ for some $R>0$ in the full group {\ca}. 
For each $n$, put $\rE_n:=\{ \{ i, j\} : 1\le i,j\le n,\ i\neq j\}$ 
and, for $\re,\rf\in\rE_n$, write $\re \sim \rf$ if $|\re \cap \rf|=1$ and 
$\re \perp \rf$ if $\re \cap \rf=\emptyset$.
One has 
\[
\Delta_n = \sum_{\re \in \rE_n} \Delta_\re,
\]
where $\Delta_{\{i,j\}}:=\sum_{r=1}^d E_{i,j}(t_r) + E_{j,i}(t_r)$.
Thus
\begin{alignat*}{3}
\Delta_n^2 
&=  \sum_{\re} \Delta_{\re}^2 & + &\sum_{\re \sim \rf} \Delta_\re\Delta_\rf&
  + &\sum_{\re \perp \rf} \Delta_\re\Delta_\rf\\
&=: \Sq_n &+ &\Adj_n& + &\Op_n.
\end{alignat*}
The elements $\Sq_n$ and $\Op_n$ are positive, 
while $\Adj_n$ is not and this causes trouble. 

For $m<n$, we view $\EL_m(\cR)$ as a subgroup 
of $\EL_n(\cR)$ sitting at the left upper corner. 
The symmetric group $\Sym(n)$ acts on $\EL_n(\cR)$ by permutation of the indices. 
We note that $|\rE_m|=\frac{1}{2}m(m-1)$, 
$|\{(\re,\rf) \in \rE_m^2 : \re\sim\rf\}|=m(m-1)(m-2)$, 
and $|\{(\re,\rf) \in \rE_m^2 : \re\perp\rf\}|=\frac{1}{4}m(m-1)(m-2)(m-3)$. 
Hence, as it is proved in \cite{kkn}, one has 
\begin{align*}
\sum_{\sigma\in \Sym(n)}\sigma(\Delta^{(2)}_m) &= m(m-1)\cdot(n-2)! \cdot\Delta^{(2)}_n,\\
\sum_{\sigma\in \Sym(n)}\sigma(\Adj_m) &= m(m-1)(m-2)\cdot(n-3)! \cdot\Adj_n,\\
\sum_{\sigma\in \Sym(n)}\sigma(\Op_m) &= m(m-1)(m-2)(m-3)\cdot(n-4)!\cdot\Op_n.
\end{align*}
Thus if we know there are $m\in\IN$, $R>0$, and $\ve>0$ such that 
\[\tag{$\heartsuit$}
\Adj_m + R \Op_m \geq \ve\Delta^{(2)}_m 
\]
holds true in $\cst[\EL_m(\cR)]$, then it follows 
\[
\frac{n-2}{m-2} \ve \Delta^{(2)}_n \le \Adj_n+\frac{m-3}{n-3}R\Op_n\le\Delta_n^2
\]
for all $n$ such that $\frac{m-3}{n-3}R\le 1$ and Main Theorem is proved. 
This is Proposition 4.1 in \cite{kkn}. 
To apply this machinery, we further expand $\Adj_m$: 
\begin{align*}
\Adj_m &= \sum_{r,s}\sum_{i,j,k;\ \mbox{\scriptsize distinct}}
  (E_{i,j}(t_r)+E_{j,i}(t_r))(E_{j,k}(t_s)+E_{k,j}(t_s))\\
 &=  \sum_{r,s} \sum_{i,j,k;\ \mbox{\scriptsize distinct}}
  \bigl( E_{i,j}(t_r)E_{j,k}(t_s) + E_{j,k}(t_s)E_{i,j}(t_r) \\
  &\hspace*{100pt} + E_{i,j}(t_r)E_{i,k}(t_s) + E_{j,k}(t_s)E_{i,k}(t_r) \bigr).
\end{align*}
Therefore, if there are $m\in\IN$, $R>0$, $\ve>0$, 
and distinct indices $i,j,k,l$ such that 
\begin{align*}\tag{$\spadesuit$}
& E_{i,j}(t_r)E_{j,k}(t_s) + E_{j,k}(t_s)E_{i,j}(t_r) \\
&\quad + E_{i,j}(t_r)E_{i,l}(t_s) + E_{j,k}(t_s)E_{l,k}(t_r) 
 + R \Op_m \geq \ve E_{i,k}(t_rt_s)
\end{align*}
holds true, then we obtain $(\heartsuit)$ (for different $R>0$ and $\ve>0$) 
by summing up this over the $\Sym(m)$-orbit and over $r,s$.
This is what we will prove in the next section. 

\section{The Heisenberg group and the rotation $\mbox{C}^*$-algebras}\label{sec:heisenberg}

In this section, we will work entirely in the {\ca} setting. 
Let's consider the \emph{integral Heisenberg group} 
\begin{align*}
\bH 
:=\{ \left[\begin{smallmatrix} 1 & a & c \\ & 1 & b \\ & & 1\end{smallmatrix}\right] : a,b,c\in\IZ\}
\cong \langle x,y : z:=[x,y]\mbox{ is central}\rangle,
\end{align*}
where 
\[
x = \left[\begin{smallmatrix} 1 & 1 &  \\ & 1 &  \\ & & 1\end{smallmatrix}\right],\,
y = \left[\begin{smallmatrix} 1 &  &  \\ & 1 & 1 \\ & & 1\end{smallmatrix}\right],\,
z = \left[\begin{smallmatrix} 1 &  & 1 \\ & 1 &  \\ & & 1\end{smallmatrix}\right].
\]
Note that every irreducible unitary representation of $\bH$ sends 
the central element $z$ to a scalar (multiplication operator) of modulus $1$. 
For $\theta\in[0,1)$, we consider the irreducible 
unitary representation $\pi_\theta$ of $\bH$ 
on $\ell_2(\IZ)$ or $\ell_2(\IZ/q\IZ)$, depending on whether $\theta$ irrational 
or $\theta=p/q$ is rational with $\gcd(p,q)=1$, given by 
\[
\pi_\theta(x)\delta_j=\exp(2j\pi \imath\theta)\delta_j,\, 
\pi_\theta(y)\delta_j=\delta_{j+1},\,
\pi_\theta(z)=\exp(2\pi \imath\theta).
\]
By convention, if $\theta=p/q$ is rational, then $\gcd(p,q)=1$ is assumed; 
and if $\theta$ is irrational, we consider $q=\infty$ and $\IZ/q\IZ$ means $\IZ$. 
Thus in either case $\pi_\theta$ is a representation on $\ell_2(\IZ/q\IZ)$. 
The {\ca} $\cA_\theta:=\pi_\theta(\cst[\bH])$ 
is called the \emph{rotation {\ca}}. 

We fix the notations that are used throughout this section. 
We denote by 
\[
X:=(1-x)^*(1-x)=2-x-x^*\in \cst[\bH]_+,\; 
X_\theta:=\pi_\theta(X)\in\cA_\theta, 
\]
and the same for $y$ and $z$. 
Note that $X+Y$ is the combinatorial Laplacian of $\bH$ w.r.t.\ 
the generating subset $\{ x^{\pm},y^{\pm}\}$, that $0\le X\le 4$, 
and that the triplets $(X_\theta,Y_\theta,Z_\theta)$, $(Y_\theta,X_\theta,Z_\theta)$, 
and $(X_{1-\theta},Y_{1-\theta},Z_{1-\theta})$ are unitarily equivalent. 
For a parameter $\lambda>0$, 
the \emph{almost Mathieu operator} on $\ell_2(\IZ/q\IZ)$ is given by 
\[
H_{\theta,\lambda}
 :=\pi_\theta(\frac{\lambda}{2}(x+x^*)+y+y^*)
 =(\lambda+2)-(\frac{\lambda}{2}X_\theta+Y_\theta).
\]
We also write $s=\sin\pi\theta$, $s_m=\sin2m\pi\theta$, and $c_m=\cos2m\pi\theta$.
In particular, 
\[
Z_\theta=2(1-\cos2\pi\theta)=4s^2. 
\]
See \cite{boca} for more information about the almost Mathieu operators 
and \cite{nitsche} for some discussion in connection with the semidefinite programming. 

Eventually, we will prove a certain inequality (Theorem~\ref{thm:formula})
about $X$, $Y$, and $Z$ (in the full group {\ca} 
of a higher-dimensional Heisenberg group) 
that leads to $(\spadesuit)$ in the previous section. 
To prove inequalities about $X$, $Y$, and $Z$, it suffices 
to work with $X_\theta$, $Y_\theta$, and $Z_\theta$ 
for each $\theta\in[0,1/2]$ separately, 
thanks to the following well-known fact (Lemma~\ref{lem:faithful}). 
The critical estimate is the one for small $\theta>0$ 
(Corollary~\ref{cor:zzz} and Lemma~\ref{lem:smalltheta}). 
The rest will work out anyway. 

\begin{lem}\label{lem:faithful}
For any dense subset $I\subset[0,1)$, the representation $\bigoplus_{\theta\in I}\pi_\theta$ 
is faithful on the full group {\ca} $\cst[\bH]$.
\end{lem}
\begin{proof}
For the readers' convenience, we sketch the proof. 
Let $\tau_\theta$ denote the tracial state on $\cst[\bH]$ associated with $\pi_\theta$. 
That is to say, if $\theta$ is irrational, then $\tau_\theta$ 
arises from the canonical tracial state on the irrational rotation {\ca} $\cA_\theta$ 
and it is given by $\tau_\theta(x^iy^j)=0$ for all $(i,j)\neq(0,0)$. 
If $\theta=p/q$ is rational, then $\tau_\theta$ is given by $\tr_q\circ\pi_\theta$, where 
$\tr_q$ is the tracial state on $\IM_q(\IC)$, and it satisfies $\tau_\theta(x^iy^j)=0$ 
for all $(i,j)\neq(0,0)$ in $(\IZ/q\IZ)^2$.
It follows that $\theta\mapsto\tau_\theta$ is continuous at irrational points 
and the assumption of the lemma implies that $\tau:=\int_0^1\tau_\theta\,d\theta$ 
is a continuous state on $\bigoplus_{\theta\in I}\pi_\theta$.
It is not hard to see that $\tau$ coincides with the tracial state associated with the 
left regular representation of $\bH$, that is to say, $\tau(x^iy^jz^k)=0$ 
for all $(i,j,k)\neq(0,0,0)$. 
Since $\bH$ is amenable, the tracial state $\tau$ is faithful 
on the full group {\ca} $\cst[\bH]$. 
\end{proof}

\begin{thm}[Boca \& Zaharescu \cite{bz}]\label{thm:bz2}
Let $\theta\in[0,1/2)$. One has 
\[
\| H_{\theta,\lambda}\| \le \lambda+2 -\frac{2\lambda}{\lambda+2}\sin\pi\theta.
\]
More precisely, for any real unit vector $\xi$ in $\ell_2(\IZ/q\IZ)$, 
\begin{align*}
\| H_{\lambda,\theta}\xi \|^2
 = \lambda^2 & + 4 
   + 2(1-\tan\pi\theta)\ip{\frac{\lambda}{2}\pi_\theta(x+x^*)\xi, \pi_\theta(y+y^*)\xi}\\
   & -\sum_m |\xi_{m-1}-\xi_{m+1}-\lambda s_m\xi_m|^2. 
\end{align*}
\end{thm}
\begin{proof}
Because the statements are formulated in a different way in \cite{bz}, 
we replicate here the proof from \cite{bz}: 
\begin{align*}
\| H_{\lambda,\theta}\xi \|^2
 &= \sum_m|\lambda c_m\xi_m + \xi_{m-1}+\xi_{m+1}|^2\\
 &= \lambda^2 + 4 +\sum_m\bigl( -\lambda^2s_m^2\xi_m^2-|\xi_{m-1}-\xi_{m+1}|^2
          +2\lambda c_m \xi_m(\xi_{m-1}+\xi_{m+1})\bigr)\\
 &=\lambda^2 + 4 -\sum_m |\xi_{m-1}-\xi_{m+1}-\lambda s_m\xi_m|^2\\
 & \qquad -2\lambda \sum_ms_m(\xi_{m-1}-\xi_{m+1})\xi_m +2\lambda \sum_m c_m \xi_m(\xi_{m-1}+\xi_{m+1}).  
\end{align*}
We continue computation:
\begin{align*}
\sum_m c_m\xi_m(\xi_{m-1}+\xi_{m+1}) 
 &=\sum_m (c_{m-1}+c_m) \xi_{m-1}\xi_m \\
 &=2\cos\pi\theta \sum_m  \xi_{m-1}\xi_m\cos(2m-1)\pi\theta
\intertext{and similarly}
 -\sum_ms_m(\xi_{m-1}-\xi_{m+1})\xi_m
 &=\sum_m (s_{m-1}-s_m) \xi_{m-1}\xi_m \\
 &=-2\sin\pi\theta \sum_m \xi_{m-1}\xi_m\cos(2m-1)\pi\theta\\
 &=-\tan\theta\sum_m c_m\xi_m(\xi_{m-1}+\xi_{m+1}). 
\end{align*}
Thus one obtains the purported formula for $\| H_{\lambda,\theta}\xi \|^2$. 
We also observe that 
\begin{align*}
\| H_{\lambda,\theta}\xi \|^2
 &\le \lambda^2 + 4  + 4\lambda(\cos\pi\theta-\sin\pi\theta)\sum_m\xi_{m-1}\xi_m\cos(2m-1)\pi\theta\\
 &\le  \lambda^2 + 4  + 4\lambda(1-\sin\pi\theta).
\end{align*}
This yields the purported estimate for $\|H_{\theta,\lambda}\|$. 
\end{proof}

\begin{cor}\label{cor:XYZ1}
In the full group {\ca} $\cst[\bH]$, one has 
\[
X+Y\geq\frac{1}{2}\sqrt{Z}.
\]
\end{cor}
\begin{proof}
By Lemma~\ref{lem:faithful}, it suffices to show the assertion in $\cA_\theta$ for each $\theta\in[0,1/2]$. 
This follows from Theorem~\ref{thm:bz2} with $\lambda=2$ that  
$X_\theta+Y_\theta=4-H_{\theta,2}\geq \frac{1}{2}\sqrt{Z_\theta}$.
\end{proof}

Since $Z$ is central, $X+Y\geq \frac{1}{2}\sqrt{Z}$ is equivalent 
to $4(X+Y)^2\geq Z$ in $\cst[\bH]$. 
However, there is no $R>0$ such that $R(X+Y)^2  \succeq Z$ in $\IR[\bH]$.
We will elaborate this in Section~\ref{sec:realgrpalg}.

\begin{cor}\label{cor}\label{cor:zzz}
Let $R\geq1$, $0<\kappa<1$, 
and $\theta_0:=\min\{\frac{1}{4},\frac{1}{\pi}\arcsin (\kappa\sqrt{\frac{1-\kappa}{R}})\}$.
Then for any $\theta\in[0,\theta_0]$, one has 
\[
R X_\theta + Y_\theta \geq \frac{\sqrt{(1-\kappa) R}}{2}\sqrt{Z_\theta}. 
\]
\end{cor}
\begin{proof}
We write $s_0:=\sin\pi\theta_0$, $c:=\diag_m c_m = \pi_\theta(\frac{x+x^*}{2}) = 1-\frac{1}{2}X_\theta$, 
and $C=\sqrt{(1-\kappa) R}$. 
Let $\theta \in [0,\theta_0]$ and a real unit vector $\xi\in\ell_2(\IZ/q\IZ)$ be given. 
We need to prove $\ip{(R X_\theta + Y_\theta)\xi,\xi} \geq Cs$. 
For this, we may assume that $\ip{\pi_\theta(x+x^*)\xi,\pi_\theta(y+y^*)\xi}>0$ 
because otherwise $\ip{(X_\theta+Y_\theta)\xi,\xi}\geq4-\|\pi_\theta(x+x^*+y+y^*)\xi\|\geq 4-2\sqrt{2}$. 
Put $\ve:=1-\| c\xi\|$. 
If $\ve\geq\frac{Cs}{2R}$, then 
$\ip{RX_\theta \xi,\xi}\geq 2R\ve\geq Cs$ and we are done. 
 From now on, we assume that $\ve<\frac{Cs}{2R}$. 
By Theorem~\ref{thm:bz2} for $\lambda:=2R/C$,
one has 
\begin{align*}
\| H_{\lambda,\theta}\xi \|^2 
 &\le \lambda^2+4+2\lambda(1-s)\ip{c\xi,(H_{\theta,\lambda}-\lambda c)\xi}\\
 &\le \lambda^2+4+2\lambda(1-s)(1-\ve)\|H_{\theta,\lambda}\xi\|-2\lambda^2(1-s)(1-\ve)^2
\end{align*}
and hence 
\begin{align*}
\bigl(\| H_{\lambda,\theta}\xi \| -\lambda(1-s)(1-\ve) \bigr)^2 
 &\le 4+\lambda^2(1-2(1-s)(1-\ve)^2+(1-s)^2(1-\ve)^2)\\
 &= 4 +\lambda^2(1-(1-s^2)(1-\ve)^2)\\
 &\le 4+ \lambda^2(s^2+2\ve). 
\end{align*}
Thus 
\[
\| H_{\lambda,\theta}\xi \| \le 2
 +\lambda s( \frac{1}{4}\lambda s_0 + \frac{1}{2}\lambda \frac{\ve}{s}) + \lambda(1-s).
\]
By our choices, $\lambda s_0 =\frac{2R}{C}\cdot\frac{\kappa\sqrt{(1-\kappa)}}{\sqrt{R}} 
=2\kappa$ and $\lambda\ve/s \le1$. 
Therefore, 
\[
\| H_{\lambda,\theta}\xi \| \le \lambda + 2
   - (1-\frac{1}{4}\cdot 2\kappa-\frac{1}{2})\cdot2\sqrt{\frac{R}{1-\kappa}}s
= \lambda + 2 - C s.
\]
Since $\lambda+2-H_{\lambda,\theta}=\frac{\lambda}{2}X_\theta+Y_\theta\le RX_\theta+Y_\theta$, we are done.
\end{proof}

\begin{prop}\label{prop:XYZ2}
In the full group {\ca} $\cst[\bH]$, one has 
\[
(X + Y)\sqrt{Z} + \frac{1}{2}(XY+YX)\geq0.
\]
\end{prop}
\begin{proof}
By Lemma~\ref{lem:faithful}, it suffices to show the same for $X_\theta$'s.
We write $b_m:=1-c_m=1-\cos2m\pi\theta = 2\sin^2 m\pi\theta$. 
We observe that 
\[
X_\theta=\left[\begin{matrix}
\ddots & & & \\
& 2b_{m-1} & & \\ 
& & 2b_{m} & \\ 
& & & \ddots
\end{matrix}\right],\quad
Y_\theta=\left[\begin{matrix}
\ddots & & & \\
& 2 & -1 & \\ 
& -1 & 2 & \\ 
& & & \ddots
\end{matrix}\right],
\]
\[
\frac{1}{2}(X_\theta Y_\theta+Y_\theta X_\theta) = 
\left[\begin{matrix}
\ddots & & & \\
& 4b_{m-1} & -(b_{m-1}+b_{m}) & \\ 
& -(b_{m-1}+b_{m})  & 4b_{m} & \\ 
& & & \ddots
\end{matrix}\right].
\]
These are the sums of the following 2-by-2 matrices 
sitting at $(m-1)$-to-$m$-th corners:
\[
X_{\theta,m}=\left[\begin{matrix}
b_{m-1} &  \\ 
&  b_{m} 
\end{matrix}\right],\quad
Y_{\theta,m}=\left[\begin{matrix}
1 & -1 \\ 
-1 & 1 
\end{matrix}\right],
\]
\[
\frac{1}{2}(XY+YX)_{\theta,m}:=\left[\begin{matrix}
2b_{m-1} & -(b_{m-1}+b_{m})\\
 -(b_{m-1}+b_{m}) &2b_{m}
\end{matrix}\right].
\]
Thus, it suffices to show 
\begin{align*}
T_{\theta,m}
&:=2s(X_{\theta,m} + Y_{\theta,m})
+\frac{1}{2}(XY+YX)_{\theta,m} \\
&= \left[\begin{matrix}
2(s +1)b_{m-1} + 2s & -(2s+b_{m-1}+b_{m})\\
-(2s+b_{m-1}+b_{m}) &2(s+1)b_{m} + 2s 
\end{matrix}\right]
\end{align*}
is positive in $\IM_2(\IC)$ for every $m$. 
We only need to calculate the determinant:
\begin{align*}
\mathrm{det}(T_{\theta,m}) 
  &\geq 4b_{m-1} b_{m}+4s(s+1)(b_{m-1}+b_{m})+4s^2
   -(2s+b_{m-1}+b_m)^2\\
 &= 4s^2(b_{m-1}+b_m)-(b_{m-1}-b_m)^2\\
 &= 8s^2 (\sin^2 (m-1)\pi\theta  + \sin^2m\pi\theta)
   - 4s^2 \sin^2 (2m-1)\pi\theta\\
&\geq0.
\end{align*}
Here, we have used the formulas $b_m=2\sin^2 m\pi\theta$, 
$b_{m-1} - b_{m}=-2s\theta\sin(2m-1)\pi\theta$, 
and $|\sin(2m-1)\pi\theta| \le |\sin (m-1)\pi\theta| + |\sin m\pi\theta|$.
\end{proof}

A similar calculation shows $Z+\frac{1}{2}(XY+YX)\geq0$ in $\cst[\bH]$. 
In fact, it is a sum of squares: 
\[
Z+\frac{1}{2}(XY+YX) = \frac{1}{4}(X+Y)Z+
\frac{1}{8}\sum (1-b)^\delta(1-a)^{\epsilon}(1-a)^{\bar{\epsilon}}(1-b)^{\bar{\delta}},
\]
where $\sum$ is over the 8 terms $(a,b)\in\{(x,y),(y,x)\}$ and $(\epsilon,\bar{\epsilon}),(\delta,\bar{\delta})\in\{(\ast,\,\cdot\,),(\,\cdot\,,\ast)\}$.

Now, we consider the {\ca} $\cA_\theta\otimes\cA_\theta$ 
on $\ell_2(\IZ/q\IZ)\otimes\ell_2(\IZ/q\IZ)$. 
We continue to view $Z_\theta$ as a scalar in $\cA_\theta\otimes\cA_\theta$. 
We want to find an inequality that leads to $(\spadesuit)$. 
The following does the job for small $\theta>0$. We note that it fails at $\theta_0=1/2$. 

\begin{lem}\label{lem:smalltheta}
There are $\theta_0>0$ , $R>1$, and $\ve>0$  
such that for every $\theta\in[0,\theta_0]$, one has 
\[
 R(X_\theta\otimes Y_\theta+Y_\theta\otimes X_\theta)
+X_\theta\otimes X_\theta+Y_\theta\otimes Y_\theta
+(X_\theta Y_\theta + Y_\theta X_\theta)\otimes 1\geq  \ve Z_\theta.
\] 
\end{lem}
\begin{proof}
By Corollary~\ref{cor:zzz}, there are $\theta_0>0$ and $R>1$  
such that $1\otimes (R X_\theta + Y_\theta) \geq 8s$ for every $\theta\in[0,\theta_0]$.
By Proposition~\ref{prop:XYZ2} and Corollary~\ref{cor:XYZ1}, it follows that 
the left hand side dominates 
\[
(X_\theta+Y_\theta)\cdot 8s
   +X_\theta Y_\theta + Y_\theta X_\theta
 \geq  (X_\theta+Y_\theta)\cdot 4s
 \geq Z_\theta,
\]
where we omitted writing $\otimes1$.  
\end{proof}

To deal with the case $\theta\geq\theta_0$, 
we need a few more auxiliary lemmas on $\cA_\theta$. 
\begin{lem}\label{lem:prodnorm}
For every $\theta\in[0,1/2]$, one has 
\[
\|\pi_\theta((1-x)(1-y))\|\le 4\cos(\pi\theta/2).
\]
\end{lem}
\begin{proof}
The expansion of $(1-y)^*(1-x)^*(1-x)(1-y)$ has 16 terms (counting multiplicity) 
and among them are 
$-(1+z)x$, $-(1+z)^*x^*$ and $x(zy^*+y)+x^*(z^*y^*+y)$.
One has $|1+z|=2\cos\pi\theta$ and 
\begin{align*}
\|x(zy^* +  &  y) + x^*(z^*y^*+y)\|
 \le \|\left[\begin{smallmatrix} xy^* & x^*y^*\end{smallmatrix}\right]\| 
   \|\left[\begin{smallmatrix} z +y^2 \\ z^*+y^2\end{smallmatrix}\right]\| \\
 &\le \sqrt{2}  \|(z+y^2)^*(z+y^2)+(z^*+y^2)^*(z^*+y^2)\|^{1/2} 
   =4\cos\pi\theta.
\end{align*}
Hence 
$\|\pi_\theta((1-y)^*(1-x)^*(1-x)(1-y))\|
\le 8 + 8\cos\pi\theta = 16\cos^2(\pi\theta/2)$.
\end{proof}

For a positive operator $A$, we denote by $\IP_{A\le\delta}$ 
(resp.\ $\IP_{A>\delta}=1-\IP_{A\le\delta}$) the spectral projection 
of $A$ corresponding to the spectrum $[0,\delta]$ (resp.\ $(\delta,\infty)$). 
We also denote by $\IP_{A\le\delta \,\wedge\, B\le \delta}$ etc.\ for 
the orthogonal projection onto $\ran\IP_{A\le\delta}\cap \ran\IP_{B\le\delta}$ etc. 
Note that if $A$ and $B$ commute, then so are their spectral projections and 
$\IP_{A\le\delta \,\wedge\, B\le \delta}=\IP_{A\le\delta}\IP_{B\le\delta}$. 

\begin{lem}\label{lem:xsmall}
For every $\theta\in (0,1/2]$ and $0<\delta<2(1-\cos\pi\theta)$, one has 
\[
\IP_{X_\theta\le\delta} Y_\theta \IP_{X_\theta\le\delta} = 2 \IP_{X_\theta\le\delta}, 
\]
the same with $X_\theta$ and $Y_\theta$ interchanged, and 
\[
\| \IP_{Y_\theta\le\delta} \IP_{X_\theta\le\delta} \| \le\sqrt{\frac{2}{4-\delta}}.
\]
In particular, $\ell_2(\IZ/q\IZ)$ is decomposed into a direct sum
\[
\ell_2(\IZ/q\IZ) = \ran  \IP_{X_\theta\le\delta} + \ran  \IP_{Y_\theta\le\delta} + 
\ran \IP_{X_\theta>\delta \,\wedge\, Y_\theta>\delta}
\]
and the corresponding (not necessarily orthogonal) projections 
have norm at most $\sqrt{\frac{4-\delta}{2-\delta}}$.
\end{lem}
\begin{proof}
We observe that $\IP_{X_\theta\le\delta}$ is the projection 
onto $\ell_2(E)$ with 
\[
E:=\{ m :  2(1-\cos2m\pi\theta) \le \delta \}
 \subset \{ m : m\theta \in (-\theta/2,\theta/2) + \IZ \}.
\]
The set $E$ does not contain consecutive numbers and the first assertion follows. 
The second follows from the unitary equivalence of the pairs 
$(X_\theta,Y_\theta)$ and $(Y_\theta,X_\theta)$. 
Since 
$Y_\theta \le \delta \IP_{Y_\theta\le\delta} + 4(1-\IP_{Y_\theta\le\delta})
 =4-(4-\delta)\IP_{Y_\theta\le\delta}$, 
one has
\[
2\IP_{X_\theta\le\delta}\le 4\IP_{X_\theta\le\delta}
  - (4-\delta)\IP_{X_\theta\le\delta}\IP_{Y_\theta\le\delta}\IP_{X_\theta\le\delta}
\]
and $\| \IP_{Y_\theta\le\delta} \IP_{X_\theta\le\delta} \|^2=\| \IP_{X_\theta\le\delta}\IP_{Y_\theta\le\delta}\IP_{X_\theta\le\delta}\|\le 2/(4-\delta)$.
This gives the desired estimate 
for $\| \IP_{Y_\theta\le\delta} \IP_{X_\theta\le\delta} \|$.
We remark that this estimate can be improved to $\approx 1/\sqrt{3}$ 
if $\theta$ is away from $1/2$ and $\delta>0$ is small enough. 
Indeed, the gaps of $E$ will have length at least two and hence 
any unit vectors $\xi\in \ran \IP_{X_\theta\le\delta}$ and 
$\eta\in\IP_{Y_\theta\le\delta}$ satisfy
\[
|\ip{\xi,\eta}| \approx |\ip{\xi,\frac{1}{3}\pi_\theta(1+y+y^*)\eta}|
=|\ip{\frac{1}{3}\pi_\theta(1+y+y^*)\xi,\eta}| 
\le1/\sqrt{3}.
\]

The projection onto the third subspace is orthogonal. 
On the other hand, 
any $\xi+\eta \in \ran \IP_{X_\theta\le\delta} + \ran \IP_{Y_\theta\le\delta}$ 
satisfies 
\[
\|\xi+\eta\|^2\geq\|\xi\|^2+\|\eta\|^2-2\| \IP_{Y_\theta\le\delta} \IP_{X_\theta\le\delta} \|\|\xi\|\|\eta\|
\geq (1-\| \IP_{Y_\theta\le\delta} \IP_{X_\theta\le\delta} \|^2) \|\xi\|^2.
\]
This gives the desired norm estimate. 
\end{proof}
Now, we consider this time the cubic tensor product $\cA_\theta\otimes\cA_\theta\otimes\cA_\theta$. 
This arises as an irreducible representation of the higher dimensional Heisenberg group 
\[
\bH_3 :=\{ 
\left[\begin{smallmatrix}
 1 & \ast & \ast & \ast & \ast \\
  & 1 & 0&0 & \ast \\
  & & 1& 0& \ast \\
& & & 1 & \ast \\
 & & & & 1\end{smallmatrix}\right]\}
\subset \SL(5,\IZ).
\]
We put $x_i:=e_{1,i+1}(1)$, $y_i:=e_{i+1,5}(1)$, and $z:=e_{1,5}(1)$ in $\bH_3$, 
where we recall that $e_{i,j}(1)$ is the elementary matrix defined 
in the beginning of the previous section. 
Note that $[x_i,y_i]=z$ and $[x_i,y_j]=1$ for $i\neq j$. 
Hence $\bH_3$ is isomorphic to the quotient of 
$\bH \times \bH \times \bH$ modulo $z$ are identified. 
As before, we write $X_i := (1-x_i)^*(1-x_i)$, etc. 
This should not be confused with $X_\theta$ in $\cA_\theta$.  

\begin{thm}\label{thm:formula}
There are $R>0$ and $\ve>0$ such that
\[
R(X_1 Y_2 + Y_1 X_2 + X_1 Y_3 + Y_1 X_3)
 + X_1 X_2 + Y_1 Y_2 + X_1 Y_1 + Y_1 X_1 \geq \ve Z
\]
holds in $\cst[\bH_3]$. 
\end{thm}
\begin{proof}
By Lemma~\ref{lem:faithful} (adapted to this case), 
it suffices to prove the assertion 
in $\cA_\theta\otimes\cA_\theta\otimes\cA_\theta$ for each $\theta\in[0,1/2]$. 
We write $X_{i,\theta}$ for $X_\theta$ in the $i$-th tensor component. 
For a unit vector 
$\zeta\in\ell_2(\IZ/q\IZ) \otimes \ell_2(\IZ/q\IZ) \otimes \ell_2(\IZ/q\IZ)$,
we need to prove 
\begin{align*}
&\ip{ \bigl( R ( X_{1,\theta} Y_{2,\theta} + Y_{1,\theta} X_{2,\theta}
 + X_{1,\theta}Y_{3,\theta} + Y_{1,\theta}X_{3,\theta})\\
&\qquad +X_{1,\theta} X_{2,\theta} + Y_{1,\theta} Y_{2,\theta} + X_{1,\theta} Y_{1,\theta} + Y_{1,\theta} X_{1,\theta}\bigr) \zeta,\zeta}\geq \ve Z_\theta.
\end{align*}
By Lemma~\ref{lem:smalltheta}, we are already done for $\theta\in[0,\theta_0]$.
For application of Lemma~\ref{lem:xsmall}, fix $0<\delta<2(1-\cos\pi\theta_0)$ 
small enough and consider $\theta\in[\theta_0,1/2]$. 
Since we may choose $R>1$ arbitrarily large with respect to the fixed $\delta$, 
we may assume 
\[
\max\{ \|\IP_{ X_{1,\theta} Y_{2,\theta}> \delta^2}\zeta\|,\, 
\|\IP_{ Y_{1,\theta} X_{2,\theta}> \delta^2}\zeta\|,\, 
\|\IP_{ X_{1,\theta} Y_{3,\theta}> \delta^2}\zeta\|,\, 
\|\IP_{ Y_{1,\theta} X_{3,\theta} > \delta^2}\zeta\|\}<\delta.
\]
As described in Lemma~\ref{lem:xsmall}, 
we consider the decomposition 
\[
\zeta=\xi+\eta+\gamma
 \in \ran\IP_{X_{1,\theta}\le\delta} + \ran\IP_{Y_{1,\theta}\le\delta}
 + \ran\IP_{X_{1,\theta}>\delta\,\wedge\, Y_{1,\theta}>\delta}.
\]
Note that $\max\{\|\xi\|,\|\eta\|,\|\gamma\|\}\le2$.
By writing $\approx_\delta$, we will mean that the difference is at most $\delta$. 
Since $\zeta\approx_\delta\IP_{ X_{1,\theta} Y_{2,\theta}\le \delta^2}\zeta$ and 
 $\IP_{Y_{2,\theta}>\delta\, \wedge\, X_{1,\theta} Y_{2,\theta} \le \delta^2}
 \le\IP_{X_{1,\theta}\le\delta\,\wedge\,Y_{2,\theta}>\delta}$, one has
\[
\IP_{Y_{2,\theta}>\delta}\zeta
\approx_{\delta} \IP_{X_{1,\theta}\le\delta\,\wedge\,Y_{2,\theta}>\delta}\zeta.
\]
It follows that
\[
\IP_{Y_{2,\theta}>\delta}\eta 
+ \IP_{Y_{2,\theta}>\delta}\gamma
 \approx_{\delta} \IP_{X_{1,\theta}\le\delta\,\wedge\,Y_{2,\theta}>\delta}(\xi+\eta+\gamma) - \IP_{Y_{2,\theta}>\delta}\xi 
 = \IP_{X_{1,\theta}\le\delta\,\wedge\,Y_{2,\theta}>\delta}\eta. 
\]
Since $\IP_{Y_{2,\theta}>\delta}$ leaves 
$\ran\IP_{X_{1,\theta}\le\delta}$ and $\ran\IP_{Y_{1,\theta}\le\delta}$ invariant, 
this implies
\[
\IP_{Y_{2,\theta}>\delta}\eta \approx_{\delta} 
  \IP_{X_{1,\theta}\le\delta\,\wedge\,Y_{2,\theta}>\delta}\eta
\ \mbox{ and }\ 
\IP_{Y_{2,\theta}>\delta}\gamma \approx_{\delta}0.
\] 
Hence, in combination with Lemma~\ref{lem:xsmall} that  
$\IP_{Y_{1,\theta}\le\delta}\IP_{X_{1,\theta}>\delta}\IP_{Y_{1,\theta}\le\delta}
\geq \frac{1}{4}\IP_{Y_{1,\theta}\le\delta}$, one obtains
$\delta^2 \geq \| \IP_{X_{1,\theta}>\delta}\IP_{Y_{2,\theta}>\delta}\eta \|^2 
 \geq\frac{1}{4}\|\IP_{Y_{2,\theta}>\delta}\eta \|^2$,
that is, 
\[
\eta \approx_{2\delta} \IP_{Y_{2,\theta}\le\delta}\eta.
\] 
The same consideration on $Y_{1,\theta}X_{2,\theta}$ yields 
\[
\IP_{X_{2,\theta}>\delta}\gamma \approx_{\delta}0 
\ \mbox{ and }\ 
\xi \approx_{2\delta} \IP_{X_{2,\theta}\le\delta}\xi.
\]
Thus $\IP_{Y_{2,\theta}>\delta}\IP_{X_{2,\theta}\le\delta}\gamma\approx_\delta
\IP_{Y_{2,\theta}>\delta}\gamma\approx_\delta 0$ and, 
by Lemma~\ref{lem:xsmall} again,
\[
\|\gamma\|^2\approx_{\delta^2}\|\IP_{X_{2,\theta}\le\delta}\gamma\|^2
\le 4 \|\IP_{Y_{2,\theta}>\delta}\IP_{X_{2,\theta}\le\delta}\gamma\|^2
\le 16\delta^2. 
\]
Further, the same for $X_{1,\theta}Y_{3,\theta}$ 
and $Y_{1,\theta}X_{3,\theta}$ yields 
\[
\xi \approx_{2\delta} \IP_{X_{3,\theta}\le\delta}\xi
\ \mbox{ and }\ 
\eta \approx_{2\delta} \IP_{Y_{3,\theta}\le\delta}\eta.
\]
Now a routine but tedious calculation with Lemma~\ref{lem:xsmall} yields
\[
\ip{X_{1,\theta}X_{2,\theta}\zeta,\zeta}
 \approx_{C\delta}\ip{X_{1,\theta}X_{2,\theta} \IP_{Y_{1,\theta}\le\delta \,\wedge\, Y_{2,\theta}\le\delta} \eta, 
   \IP_{Y_{1,\theta}\le\delta \,\wedge\, Y_{2,\theta}\le\delta}\eta} 
\approx_{16\delta}4\|\eta\|^2
\]
for some absolute constant $C$ (e.g., $C=1000$ should be enough),
and likewise 
\[
\ip{Y_{1,\theta}Y_{2,\theta}\zeta,\zeta} \approx_{C\delta} 4\|\xi\|^2.
\]
On the other hand, by Lemmas~\ref{lem:prodnorm} and \ref{lem:xsmall}, 
\begin{align*}
&|\ip{(X_{1,\theta}Y_{1,\theta}+Y_{1,\theta}X_{1,\theta})\zeta,\zeta}|\\
 &\qquad \approx_{C\delta} 2|\ip{X_{1,\theta}Y_{1,\theta} 
   \IP_{X_{1,\theta}\le\delta \,\wedge\, X_{2,\theta}\le\delta \,\wedge\, X_{3,\theta}\le\delta}\xi,
  \IP_{Y_{1,\theta}\le\delta \,\wedge\, Y_{2,\theta}\le\delta \,\wedge\, Y_{3,\theta}\le\delta}\eta}|\\
&\qquad\le 2 \| \IP_{Y_{1,\theta}\le\delta}\pi_\theta(1-x_1^*)\|
  \|\pi_\theta((1-x_1)(1-y_1))\| \|\pi_\theta(1-y_1^*)\IP_{X_{1,\theta}\le\delta}\|\\
 &\qquad\qquad\times \| \IP_{X_{2,\theta}\le\delta}\IP_{Y_{2,\theta}\le\delta}\|
  \| \IP_{X_{3,\theta}\le\delta}\IP_{Y_{3,\theta}\le\delta}\| \|\xi\|\|\eta\|\\
&\qquad\le 16 (\cos\frac{\pi\theta}{2})\cdot\frac{2}{4-\delta}\|\xi\|\eta\|.
\end{align*}
If we have chosen $\delta>0$ small enough, 
then 
\[
\ve:=8-16 (\cos\frac{\pi\theta_0 }{2})\cdot\frac{2}{4-\delta} > 4C\delta.
\]
Observe that $\delta>0$ and $\ve>0$ depends on 
the absolute constants $\theta_0>0$ and $C>0$, 
but not on $\theta\in[\theta_0,1/2]$. 
In the end, 
\begin{align*}
|\ip{(X_{1,\theta}Y_{1,\theta}+X_{1,\theta}Y_{1,\theta})\zeta,\zeta}| 
 &\le (8-\ve)\|\xi\|\eta\| + C\delta\\
 &\le 4(1-\ve/2)(\|\xi\|^2+\|\eta\|^2)+C\delta\\
 &\le \ip{(X_{1,\theta}X_{2,\theta}+Y_{1,\theta}Y_{2,\theta})\zeta,\zeta}
  -\ve + 3C\delta.
\end{align*}
This completes the proof. 
We remark that the above proof for $\theta\in[\theta_0,1/2]$ 
is not as tight as it appears (and $\ve>0$ can be ``visible''), 
because if $\theta$ is around $1/2$, 
then $\cos\frac{1}{2}\pi\theta\approx1/\sqrt{2}$; and 
if $\theta$ is away from $1/2$, then 
$\|\IP_{X_{\theta}\le\delta}\IP_{Y_{\theta}\le\delta}\|$ 
is bounded by $\approx1/\sqrt{3}$.
\end{proof}
\section{Proof of Main Theorem, Postlude}\label{sec:postlude}

Since $\cR:=\IZ\ip{t_1,\ldots,t_d}$ is \emph{commutative}, 
we may apply Theorem~\ref{thm:formula} to 
$x_1=e_{1,2}(t_r)$, $x_2=e_{1,3}(t_s)$, $x_3=e_{1,4}(t_r)$, 
$y_1=e_{2,5}(t_s)$, $y_2=e_{3,5}(t_r)$, $y_3=e_{4,5}(t_s)$, 
and $z=e_{1,5}(t_rt_s)$ in $\EL_5(\cR)$. 
This yields $(\spadesuit)$ in Section~\ref{sec:prelude} and 
the proof of Main Theorem is complete. 
\hspace{\fill}\qedsymbol
\smallskip

The terms $X_1Y_2 = E_{1,2}(t_r)E_{3,5}(t_r)$ 
and $Y_1X_2=E_{2,5}(t_s)E_{1,3}(t_s)$ are 
diagonal w.r.t.\ $\{t_r, t_s\}$. This causes an annoying 
dependence of $R$ on $d$ in the formula $(\heartsuit)$, 
which results in dependence of $n_0$ of $d$ in Main Theorem. 

\section{Real group algebras and property $\mathrm{H}_{\mathrm{T}}$}\label{sec:realgrpalg}
In this section, we continue the study of \cite{nt,nt2,nitsche,cec,ncrag} about positivity 
in real group algebras. 
In addition to the notations from Section~\ref{sec:prelim}, 
we denote by 
\[
I[\Ga] :=\lh\{ 1-x : x\in \Ga\} \subset \IR[\Ga]
\] 
the \emph{augmentation ideal}. 
We observe that $\Sigma^2I[\Ga] = I[\Ga] \cap \Sigma^2\IR[\Ga]$ and 
hence there is no ambiguity about the order $\preceq$ on $I[\Ga]$. 
In \cite{ncrag}, it was observed that the combinatorial Laplacian 
$\Delta \in \Sigma^2I[\Ga]$
is an \emph{order unit} for $I[\Ga]$ (more precisely for $I[\Ga]^\her$, but this abuse 
of terminology should not cause any problem). That is to say, 
for every $\xi\in I[\Ga]^\her$, there is $R>0$ 
such that $\xi\preceq R\Delta$. We will indicate this by the notation $\xi\pprec\Delta$. 

We review the relation between positive linear functionals 
on $I[\Ga]$ and \emph{$1$-cocycles} (with unitary coefficients). 
A linear functional $\vp$ on $I[\Ga]$ is said to be \emph{positive} if 
it is selfadjoint and $\vp(\Sigma^2I[\Ga])\subset\IR_{\geq0}$. 
One has $\vp(\Delta)=0$ if and only if $\vp=0$. 
Every positive linear functional $\vp$ gives rise to a semi-inner product 
$\ip{\xi,\eta}:=\vp(\xi^*\eta)$ and the corresponding semi-norm 
$\|\xi\|:=\vp(\xi^*\xi)^{1/2}$ on $I[\Ga]$, with respect to which 
the left multiplication by an element of $\Ga$ is orthogonal. 
This is the Gelfand--Naimark construction. 
The map $b\colon\Ga\ni t\mapsto 1-t\in I[\Ga]$ is a $1$-cocycle, i.e., it satisfies 
$b(st)=b(s)+sb(t)$ for every $s,t\in\Ga$.  
We note that $\vp(1-t)=\frac{1}{2}\vp((1-t)^*(1-t))=\frac{1}{2}\| b(t)\|^2$ and  
$\vp(\Delta)=\frac{1}{2}\sum_{s\in S} \|b(s)\|^2$. 
In fact, every $1$-cocycle arises in this way. 
See, e.g., Appendix C in \cite{bhv} and Appendix D in \cite{bo} for a comprehensive treatment. 

It is proved in \cite{ncrag} that $\overline{\Sigma^2 I[\Ga]} = I[\Ga]^\her\cap\overline{\Sigma^2\IR[\Ga]}$. That is to say, 
\begin{align*}
\overline{ \Sigma^2I[\Ga] }
 &:= \{ \xi \in I[\Ga]^\her : \forall\ve>0\ \xi+\ve\Delta \succeq0\}\\
 &= \{ \xi \in I[\Ga]^\her : \vp(\xi) \geq0\mbox{  for every positive linear functional $\vp$ on $I[\Ga]$}\}\\
 &= \{ \xi \in  I[\Ga]^\her : \xi\geq0\mbox{ in }\cst[\Ga]\}
\end{align*}
We also record an easy consequence of the Hahn--Banach 
separation theorem (a.k.a.\ Eidelheit--Kakutani separation theorem in this context). 
For $\xi,\eta\in I[\Ga]^\her$ (or in any real ordered vector 
space with an order unit $\Delta$), the following are equivalent. 
\begin{enumerate}
\item $\vp(\xi)=0\Rightarrow\vp(\eta)\le0$ for every positive linear functional $\vp$ on $I[\Ga]$. 
\item $-\eta \in \overline{ \Sigma^2I[\Ga] - \IR\xi }$. 
\item $\forall\ve>0$ $\exists R\in\IR$ such that $R\xi-\eta + \ve\Delta \succeq 0$. 
\end{enumerate}

We observe that since 
\[
\vp(\Delta^2) = \ip{\Delta,\Delta} = \|\sum_{s\in S} b(s) \|^2.
\]
one has $\vp(\Delta^2)=0$ if and only if the corresponding 
$1$-cocycle $b$ is \emph{harmonic} in the sense $\sum_{s\in S}b(s)=0$. 
This observation recovers Shalom's theorem (\cite{shalom:rigidity}) 
that every finitely generated group without property (T) 
has a non-zero harmonic $1$-cocycle. An essentially same proof 
was given in \cite{nitsche}. 

We record the following well-known fact.  
\begin{itemize}
\item
If a $1$-cocycle $b$ vanishes on 
a normal subgroup $N\triangleleft\Ga$, 
then $N$ acts trivially on $\lh b(\Ga)$ and 
hence $b$ factors through the quotient $\Ga/N$.  
\item
If $b$ is a harmonic $1$-cocycle on $\Ga$,  
then the center $\cZ(\Ga)$ acts trivially on 
$\lh b(\Ga)$ and $\Ga$ acts trivially on $\lh b(\cZ(\Ga))$. 
\item 
Every harmonic $1$-cocycle on an abelian group is an additive homomorphism. 
\end{itemize}
The first assertion is not difficult to show. The second follows from the identity 
$(1-x)b(z)=(1-z)b(x)$ for $x\in\Ga$ and $z\in\cZ(\Ga)$. 
If $b$ is harmonic, then $(|S|-\sum_{s\in S}s)b(z)=0$ and, 
by strict convexity of a Hilbert space, $b(z)=sb(z)$ for $s\in S$ 
and hence for all $s\in\Ga$. 

An additive character $\chi\colon\Ga\to\IR$ can be viewed as 
a harmonic $1$-cocycle. The corresponding positive linear functional 
$\vp_\chi\colon I[\Ga]\to\IR$ is given by $\vp_\chi(1-t)=\frac{1}{2}\chi(t)^2$. 
This should not be confused with the linear extension $\chi\colon I[\Ga]\to\IR$ 
which is not even selfadjoint. 
The positive linear functional $\vp_\chi$ 
factors through the abelianization $I[\Ga^\ab]$. 

We denote the \emph{augmentation power} by  
\[
I^k[\Ga] := \lh (I[\Ga]^k ) \subset \IR[\Ga].
\]
It is well-known and easy to see from the formula 
\[
1-xy = (1-x) + (1-y)  - (1-x)(1-y) \in (1-x) + (1-y) + I^2[\Ga]
\] 
that $I[\Ga]$ is generated as a rng by $\{ 1-s : s\in S\}$ 
and that $\Ga\ni x \mapsto 1-x\in I[\Ga]/I^2[\Ga]$ is 
an additive homomorphism. 
On the other hand, every additive homomorphism 
$\chi$ vanishes on $I^2[\Ga]$, 
because $\chi((1-x)(1-y))=\chi(1-x-y+xy)=0$. 
Hence $I^2[\Ga]=\bigcap_\chi \ker\chi$, where the intersection is taken 
over the additive characters $\chi$ on $\Ga$. 
We will see that $\Delta^2\in\Sigma^2I^2[\Ga]$ need not 
be an order unit for $I^4[\Ga]$, but the following element
\[
\square:=\frac{1}{4}\sum_{s,t \in S}(1-s)^*(1-t)^*(1-t)(1-s) \in \Sigma^2 I^2[\Ga]
\]
is. Since $\square=\Delta^2$ in $I[\Ga^\ab]$, 
one has $\vp_\chi(\square)=\vp_\chi(\Delta^2)=0$ for every 
additive character $\chi$. 
We will prove later that the converse is also true.

\begin{thm}\label{thm:sqher} 
The element $\square$ is an order unit for $I^4[\Ga]$. Namely
\[
I^4[\Ga]^\her=
\{ \xi \in \IR[\Ga]^\her : \pm\xi\pprec\square \}
=\lh\Sigma^2I^2[\Ga]
\] 
and moreover $I^4[\Ga]\cap\Sigma^2\IR[\Ga]=\Sigma^2I^2[\Ga]$. 
\end{thm}
\begin{proof}
We first prove that the left is contained the middle. 
The proof is similar to that for Lemma~2 in \cite{ncrag}.
Since $\xi^*\eta+\eta^*\xi \preceq \xi^*\xi+\eta^*\eta$
for every $\xi,\eta$, it suffices to show that 
$(1-x)^*(1-y)^*(1-y)(1-x) \pprec \square$ for all $x,y\in\Ga$.
By using the inequality
\begin{align*}
&(1-x_1x_2)^*(1-y)^*(1-y)(1-x_1x_2)\\ 
&\qquad= ((1-x_1)+x_1(1-x_2))^*(1-y)^*(\rule[2.5pt]{20pt}{1pt})\\
&\qquad\preceq 2 (1-x_1)^*(1-y)^*(\rule[2.5pt]{20pt}{1pt})+2(1-x_2)^*(1-x_1^{-1}yx_1)^*(\rule[2.5pt]{20pt}{1pt}),
\end{align*}
one can reduce this to the case $x\in S$, and similarly to the case $y\in S$, 
where the assertion is obvious.
We next show that $\pm\xi\pprec\square$ implies 
$\xi\in \lh\Sigma^2I^2[\Ga]$. 
There is $R>0$ such that $0\preceq R\square-\xi\preceq 2R\square$. 
Thus it remains to show $\sum_i\eta_i^*\eta_i\pprec\square$ 
implies $\eta_i\in I^2[\Ga]$. 
Since $\vp_\chi(\square)=0$ for every additive character $\chi$ on $\Ga$, 
one has 
\[
0=\vp_\chi(\sum_i\eta_i^*\eta_i)=-\frac{1}{2}\sum_{i,x,y}\eta_i(x)\eta_i(y)\chi(x^{-1}y)^2
=\sum_i\bigl(\sum_x\eta_i(x)\chi(x)\bigr)^2,
\]
or equivalently $\eta_i\in\bigcap_{\chi}\ker\chi=I^2[\Ga]$ for all $i$. 
\end{proof}

\begin{cor}
A positive linear functional $\vp$ on $I[\Ga]$ satisfies 
$\vp(\square)=0$ if and only if the associated $1$-cocycle is 
an additive homomorphism. 
\end{cor}
\begin{proof}
We have already noted that $\vp_\chi(\square)=0$ 
for all additive character $\chi$. 
Conversely, suppose $\vp(\square)=0$. 
Since this implies $\vp(\Delta^2)=0$, the $1$-cocycle $b$ associated 
with $\vp$ is harmonic. 
Moreover, since 
\[
1-[x,y]=(xy-yx)x^{-1}y^{-1}=((1-x)(1-y)-(1-y)(1-x))x^{-1}y^{-1} \in I^2[\Ga],
\]
Theorem~\ref{thm:sqher} implies that $b=0$ 
on the commutator subgroup $[\Ga,\Ga]$. 
Thus $b$ factor through $\Ga^\ab$ and is an additive homomorphism.
\end{proof}

We recall that a finitely generated group $\Ga$ is said to have 
\emph{Shalom's property $\mathrm{H}_{\mathrm{T}}$} 
if every harmonic $1$-cocycle on $\Ga$ is an additive homomorphism. 
Property $\mathrm{H}_{\mathrm{T}}$ coincides with 
Kazhdan's property (T) for groups with finite abelianization. 
It is observed in \cite{shalom:hfd} that 
finitely generated nilpotent groups have property $\mathrm{H}_{\mathrm{T}}$. 
We conjecture that the group $\EL_n(\IZ\langle t_1,\ldots,t_d\rangle)$ has 
property $\mathrm{H}_\mathrm{T}$. 
By the Hahn--Banach separation theorem, one obtains 
the following characterization of property $\mathrm{H}_{\mathrm{T}}$, 
which does not seem useful though.  
\begin{cor}
The finitely generated group $\Ga$ has finite abelianization 
if and only if $\Delta\pprec\square$.
The finitely generated group $\Ga$ has 
property $\mathrm{H}_{\mathrm{T}}$ if and only if 
for every $\ve>0$ there is $R>0$ such that 
$\square\preceq R\Delta^2+\ve\Delta$.
\end{cor}

Property $\mathrm{H}_{\mathrm{T}}$ for nilpotent groups 
also follows from Corollary~\ref{cor:XYZ1} that 
if a commutator $z=[x,y]$ is central, 
then $(1-z)^*(1-z) \ll \Delta^2$ in $\cst[\Ga]$. 
It is tempting to conjecture that every finitely generated nilpotent 
group $\Ga$ satisfies $\square\ll\Delta^2$. 
Had it been true that $\square\pprec\Delta^2$ for a given group $\Ga$, 
it would have been able to rigorously prove this by computer calculations 
because $\square$ is an order unit for $I^4[\Ga]$ (modulo a quantitative 
estimate, see \cite{nt2}). 
However, we will observe here that $\square\not\pprec\Delta^2$ in $\IR[\bH]$. 
Hence, unlike property (T), property $\mathrm{H}_{\mathrm{T}}$ is probably 
not characterized by a ``simple"\footnote{The quantifier elimination techniques, 
which the author is not familiar with, may be relevant.} 
inequality in the real group algebra. 
This spoils the current methods of proving something like Main Theorem by 
computer calculations. 
(Note that $\EL_n(\IZ\langle t\rangle)$ has the Heisenberg group $\bH_{n-2}$ 
as a quotient and the analogous statement to the following proposition holds 
true for this group.)

\begin{prop}\label{prop:nonpositivity}
Let $\bH$ be the integral Heisenberg group and $z:=[x,y]$ be 
as described in the beginning of Section~\ref{sec:heisenberg}. 
Then $(1-z)^*(1-z) \not\pprec\Delta^2$ in $\IR[\bH]$. 
Moreover, 
\[
\overline{\Sigma^2 I^2[\bH]} \neq I^4[\bH]^\her \cap \overline{\Sigma^2\IR[\bH]}.
\]
\end{prop}

The proof of $\overline{\Sigma^2 I[\Ga]} = I[\Ga]^\her \cap \overline{\Sigma^2\IR[\Ga]}$ 
given in \cite{ncrag} is based on Schoenberg's theorem that any positive linear functional 
on $I[\Ga]$ is approximable by those that extend on $\IR[\Ga]$. 
The above proposition says there is no good enough analogue of Schoenberg's theorem 
for augmentation powers. 
For the proof of the proposition, 
we need a description of the graded vector 
space $\cdots\supset I^4[\bH]\supset  I^5[\bH]\supset\cdots$. 
To ease notation, we write $\ux := 1-x$ etc.\ and observe that 
$\uz \in \cZ(\IR[\bH]) \cap I^2[\bH]$ and 
\[
\uy \ux = \ux \uy + \uz -\uz \ux - \uz \uy +\uz \uy \ux \in
 %\ux \uy + \uz + \uz I[\bH] \subset 
  \ux \uy + \uz + I^3[\bH].
\]

\begin{lem}\label{lem:linindep}
For every $n\in\IN$, the set $\{ \ux^i \uy^j \uz^k + I^n[\bH] : i,j,k\geq0,\,i+j+2k<n\}$ forms a basis for $\IR[\bH]/I^n[\bH]$. 
In particular 
\[
\dim  I^n[\bH]/I^{n+1}[\bH]=(\lfloor n/2\rfloor +1)(n-\lfloor n/2\rfloor +1). 
\]
\end{lem}
\begin{proof}
We first observe that the asserted set spans $\IR[\bH]/I^n[\bH]$. 
Indeed, this follows from the above equation for $\uy \ux$ 
and the general fact that 
$1-uv = (1-u)+(1-v)-(1-u)(1-v)$ and $1-u^{-1} = -(1-u)+(1-u^{-1})(1-u)$ 
for every $u, v\in\bH$.
It is left to show that the asserted set is also linearly independent. 
Suppose that  $\xi:=\sum_{ i+j+2k<n } \alpha_{i,j,k} \ux^i \uy^j \uz^k \in I^n[\bH]$. 
By considering the abelianization $\pi^\ab\colon \cst[\bH]\to\cst[\IZ^2]$, 
one sees $\alpha_{i,j,k}=0$ whenever $k=0$. 
It follows that $\xi\in I^n[\bH] \cap \uz\IR[\bH]$. 
We claim that $I^n[\bH]\cap \uz \IR[\bH]=\uz I^{n-2}[\bH]$ for $n\geq2$. 
Since $\uz$ is not a zero divisor in $\IR[\bH]$ (e.g., because 
$\pi_\theta(\uz)$ are invertible for $\theta\in(0,1)$), 
Lemma would follow from this claim by induction. 

The homomorphisms $\IR[\langle x \rangle] \hookrightarrow \IR[\bH]$ and 
$\IR[\langle y \rangle] \hookrightarrow \IR[\bH]$ extend to a linear injection 
$\sigma\colon \IR[\langle x \rangle] \otimes \IR[\langle y \rangle] 
\hookrightarrow \IR[\bH]$, $\xi\otimes\eta\mapsto \xi\eta$, with the left inverse 
$\pi^\ab\colon \IR[\bH]\to\IR[\IZ^2]\cong \IR[\langle x \rangle] \otimes \IR[\langle y \rangle]$. 
Since $\uy \ux \in \ux \uy + \uz\IR[\bH]$ and likewise for $\ux^*$ 
and $\uy^*$ (thanks to suitable symmetries $x\leftrightarrow x^{-1}$ and 
$y\leftrightarrow y^{-1}$ on $\bH$), one has 
--\[
I^n[\bH]\cap\uz\IR[\bH] \subset (\ran\sigma + \uz I^{n-2}[\bH])\cap\ker\pi^\ab=\uz I^{n-2}[\bH].
\]
This proves the claim.
\end{proof}

\begin{proof}[Proof of Proposition~\ref{prop:nonpositivity}]
We observe that in $I^4[\bH]/I^{5}[\bH]$
\begin{align*}
( \ux \ux \uy \uy )^* &= \uy \uy \ux \ux = \uy\ux\uy\ux+\uy\ux\uz
 = \ux \uy \ux \uy + 3\ux\uy\uz+2\uz\uz = \ux\ux\uy\uy + 4\ux\uy\uz+2\uz\uz. 
\end{align*}
We define a linear functional $\vp$ on $I^4[\bH]/I^{5}[\bH]$ by 
$\vp(\ux^4)=\vp(\uy^4)=1$, $\vp(\uz^2)=-2$,  $\vp(\ux^2\uy^2)=-1$, 
$\vp(\ux\uy\uz)=1$, and zero on all the other basis elements. 
Then, the linear functional $\vp$ is selfadjoint. Moreover, 
with respect to the basis $\{\ux\ux,\ux\uy,\uy\ux,\uy\uy\}$ for $I^2[\bH]/I^3[\bH]$, 
the bilinear form $(\xi,\eta)\mapsto\vp(\xi^*\eta)$ is represented by the matrix 
\[
\left[\begin{matrix}
 1 & 0 & 0 & -1\\ 0 & 1 & 0 & 0\\ 0 & 0 & 1 & 0 \\ -1 & 0 & 0 & 1
\end{matrix}\right].
\]
Since this matrix is positive semidefinite, the linear functional is positive on $I^4[\bH]$, 
by Theorem~\ref{thm:sqher}. 
One sees that $\vp(\uz^*\uz)=-\vp(\uz\uz)=2>0$, $\vp(\square)=4$, and 
\[
\vp(\Delta^2)=\vp((\ux\ux+\uy\uy)(\ux\ux+\uy\uy))=0.
\]
Therefore there cannot be $R>0$ such that $\uz^*\uz\preceq R\Delta^2 + \frac{1}{4}\square$. 
It follows that $4\Delta^2-\uz^*\uz \notin \overline{\Sigma^2 I^2[\bH]}$, while 
$4\Delta^2-\uz^*\uz\in  I^4[\bH]^\her \cap \overline{\Sigma^2\IR[\bH]}$ by Corollary~\ref{cor:XYZ1}.
\end{proof}
\section{Property $(\tau)$}\label{sec:tau}
We say a finitely generated group $\Ga=\langle S \rangle$ has \emph{property $(\tau)$} 
with respect to a family $\{ \Ga_i\}$ of finite quotients $\Ga\twoheadrightarrow\Ga_i$ if 
there is $\delta>0$ such that any unitary representation 
$\pi$ of $\Ga$ that factors through some $\Ga\twoheadrightarrow \Ga_i$ 
either admits a nonzero $\pi(\Ga)$-invariant vector or admits 
no unit vector $v$ such that $\max_{s\in S}\| v - \pi(s) v\| \le\delta$.
This is equivalent to that the Cayley graphs of $\{\Ga_i\}$ w.r.t.\ 
the generating subset $S$ form an expander family. 
In case the family $\{\Ga_i\}_i$ is the set of all finite quotients of $\Ga$, 
it is simply said $\Ga$ has property $(\tau)$. 
See \cite{kowalski} for a comprehensive treatment of expander graphs. 
By Main Theorem, $\EL_n(\cS)$ has property (T) if $\cS$ is a finitely generated 
\emph{irng} (i.e., a rng which is idempotent, $\cS=\cS^2$, see \cite{mot}) 
and $n$ is large enough. 
Corollaries~\ref{cor:kaz} and \ref{cor:tau} say this happens uniformly for \emph{finite} commutative irngs 
with a fixed number of generators. 

\begin{proof}[Proof of Corollary~\ref{cor:kaz}]
Let $n_0$ be as in Main Theorem for $\IZ\langle T_1,\ldots,T_d,S_1,\ldots,S_d\rangle$
 and $n\geq n_0$. 
By Main Theorem applied to 
$T_r\mapsto t_r^k$ and $S_r\mapsto t_r^{k+1}$, 
there is $\ve>0$ such that 
\[
\Delta_k:=\sum_{i\neq j}\sum_{r=1}^d (1-e_{i,j}(t_r^k))^*(1-e_{i,j}(t_r^k)) \in \IR[\EL_n(\IZ\langle t_1,\ldots,t_d\rangle)]
\]
(so $\Delta_1=\Delta$) satisfy
\[
(\Delta_k + \Delta_{k+1})^2 \geq \ve (\Delta_{2k}+\Delta_{2k+1}+\Delta_{2k+2})
\]
for all $k$. 
We may also assume that $\ve>0$ satisfies $\Delta_1^2\geq\ve\Delta_2$. 

Let $\pi,\cH$ and $v$ be given for $\EL_n(\IZ\langle t_1,\ldots,t_d\rangle)$ 
(but we will omit writing $\pi$ to ease notation) and put 
\[
\delta:= (\sum_{i,j,r} \| v - e_{i,j}(t_r) v\|^2)^{1/2}
 = \ip{\Delta v,v}^{1/2}. 
\]
We assume $\delta<(1/2)^{10}$ and put $\rho:=\delta^{1/10}$. 
Recall the notation 
that $\IP_{\Delta\le (\delta/\rho)^2}$ stands for the spectral projection 
of $\Delta$ for the interval $[0,(\delta/\rho)^2]$. 
For $v_0:=\IP_{\Delta \le (\delta/\rho)^2} v$, one has $\|v-v_0\|\le\rho$ and 
\[
\ip{(\Delta_1+\Delta_2)v_0,v_0} \le \delta^2+\ve^{-1}(\delta/\rho)^4 =:\delta_0^2.
\]

Now, $v_1:=\IP_{\Delta_1+\Delta_2 \le (\delta_0/\rho^2)^2} v_0$ satisfies 
$\|v_0-v_1\|\le\rho^2$ and 
\[
\ip{(\Delta_2+\Delta_3)v_1,v_1} \le \ve^{-1}(\delta_0/\rho^2)^4=:\delta_1^2.
\]
We continue this and obtain $v_2:=\IP_{\Delta_2+\Delta_3 \le (\delta_1/\rho^3)^2} v_1,\ldots$ 
such that 
$\| v_k-v_{k+1} \|\le \rho^{k+2}$ and 
\[
\ip{(\Delta_{2^k}+\Delta_{2^k+1})v_k,v_k} \le \ve^{-1}(\delta_{k-1}/\rho^{k+1})^4=:\delta_k^2.
\]
Then the vector $w:=\lim_k v_k$ satisfies $\|v_k-w\|\le\rho^{k+1}$ (as $\rho<1/2$). 
Moreover,  
\begin{align*}
2^{-k} |\log\delta_k| &= 2^{-(k-1)}|\log\delta_{k-1}| -2^{-(k-1)}(k+1)|\log\rho| + 2^{-(k+1)} \log\ve\\
 &= |\log\delta_0| - \bigl(\sum_{m=1}^k 2^{-(m-1)} (m+1)\bigr) |\log\rho| + \frac{1}{2}(1-2^{-k})\log\ve\\
 &> |\log\delta|/10 
\end{align*}
if $\delta>0$ is small enough compared to $\ve>0$. 
Hence $\delta_k\to0$ at a double exponential rate.

We need to show $\lim_{l} \max_{i,j,r} \| w - e_{i,j}(t_r^l) w\| =0$.
We first observe that 
\[
\| w - e_{i,j}(t_r^{2^k}) w\| \le 2\| v_k-w\| + \delta_k \le \rho^{k}+\delta_k
\]
Let $l$ be given. Take $k=k(l)$ such that $l\in[2^k,2^{k+1})$ and write 
$l=  2^k+\sum_{m=0}^{k-1} a(m) 2^m$ with $a(m)\in\{0,1\}$. 
Then for $b:=\sum_{m=0}^{\lfloor k/2\rfloor-1} a(m)2^m$,
one has 
\[
\| e_{i,j}(t_r^l) w - e_{i,j}(t_r^{2^k+b}) w \|
 \le \sum_{m=\lfloor k/2\rfloor}^{k-1} a(m)(\rho^m+\delta_m),
\]
which tends to $0$ as $l\to\infty$. 
Observe that the recurrence relation $p_0:= 2^{k-\lfloor k/2\rfloor}$ 
and $p_{m+1} := 2p_m + a(\lfloor k/2\rfloor-1-m)$ gives 
$p_{\lfloor k/2\rfloor}=2^k+b$. 
Now by arguing as in the previous paragraph, 
but starting at $v_{k-\lfloor k/2\rfloor}$ and using 
$(\Delta_{p_m}+\Delta_{p_m+1})^2\geq \ve(\Delta_{p_{m+1}}+\Delta_{p_{m+1}+1})$, 
one obtains %$w'$ instead of $w$ and 
\[
\| v_{k-\lfloor k/2\rfloor} - e_{i,j}(t_r^{2^k+b}) v_{k-\lfloor k/2\rfloor} \|
 \le \rho^{k-\lfloor k/2\rfloor} + \delta_k \to 0.
\]
Since $\| v_{k-\lfloor k/2\rfloor}-w\|\to0$ as $l\to\infty$, this completes the proof. 
\end{proof}

We give a proof of the remark that was made after Corollary~\ref{cor:kaz}.
Let $\cR:=\IZ\langle t_1,\ldots,t_d\rangle$. 
Since $\EL_n(\cR/\cR^l)$ is nilpotent, 
there is a \emph{proper} $1$-cocycle $b_l$ (see Section 2.7 in \cite{bhv} or Section 12 in \cite{bo}). 
We view $b_l$ as $1$-cocycles on $\EL_n(\cR)$ and 
consider $b:=\sum_l^{\oplus} b_l$, which we may 
assume convergent pointwise on $\EL_n(\cR)$. 
We denote by $\pi_k$ the Gelfand--Naimark representation 
associated with the positive definite function $\vp_k(x):=\exp(-\frac{1}{k}\|b(x)\|^2)$.
Then, the representation $\pi:=\bigoplus\pi_k$ 
simultaneously admits asymptotically invariant vectors and 
a weak operator topology null sequence $x_l \in\EL_n(\cR^l)$. 

\begin{proof}[Proof of Corollary~\ref{cor:tau}]
Let $\cR^1:=\IZ[t_1,\ldots,t_d]$ denote 
the unitization of $\cR:=\IZ\langle t_1,\ldots,t_d\rangle$. 
Any quotient map $\cR\twoheadrightarrow\cS$ with $\cS$ unital 
gives rise to a group homomorphism 
$\EL_n(\cR^1)\twoheadrightarrow\EL_n(\cS)$ that 
extends $\EL_n(\cR)\twoheadrightarrow\EL_n(\cS)$.
We need to show that an orthogonal representation of $\EL_n(\cR^1)$ 
which factors through $\EL_n(\cS)$ has a nonzero invariant 
vector, provided that it has almost $\EL_n(\cR)$ invariant vector. 
Since we know $\EL_n(\cR^1)$ has property (T), 
it suffices to show that every almost $\EL_n(\cR)$ invariant 
vector is also almost $\EL_n(\IZ1)$ invariant. 
The latter is true when $\cS$ is finite.
Indeed, the vector $w$ in Corollary~\ref{cor:kaz} 
is invariant under those $e_{i,j}(t_r^{l_0})$ 
such that $t_r^{l_0}$ is an idempotent in the quotient $\cS$.  
Since a finite commutative ring is a direct sum of local rings 
(see, e.g., \cite{kn}), the rng generated by such idempotents 
contains the identity of $\cS$ and hence $w$ is invariant under $\EL_n(\IZ1)$. 
\end{proof}

\end{document}